\documentclass[a4paper, 12pt, reqno]{amsart}
\usepackage[utf8]{inputenc}

\usepackage{amsmath}
\usepackage{amsfonts}
\usepackage{amssymb}
\usepackage{lipsum}
\usepackage{physics}
\usepackage{esint}

\usepackage{amsthm}
\usepackage{enumerate}
\usepackage{amsrefs}

\usepackage{graphicx}
\usepackage{tikz}

\usepackage{threeparttable}
\usepackage{float} 

\usepackage[titletoc]{appendix}
\usepackage[english]{babel}
\usepackage[utf8]{inputenc}
\usepackage[none]{hyphenat}

\usepackage{geometry}
\usepackage{hyperref}
\hypersetup{
    colorlinks=true,
    linkcolor=blue,
    citecolor=blue,
    filecolor=magenta,      
    urlcolor=cyan,
}

\newcommand\blfootnote[1]{%
  \begingroup
  \renewcommand\thefootnote{}\footnote{#1}%
  \addtocounter{footnote}{-1}%
  \endgroup
}

\newcounter{foo}

\theoremstyle{plain}
\newtheorem{conj}[foo]{Conjecture}
\newtheorem{thm}[foo]{Theorem}
\newtheorem{prop}[foo]{Proposition}
\newtheorem{lem}[foo]{Lemma}

\theoremstyle{definition}
\newtheorem{defn}[foo]{Definition}

\theoremstyle{remark}
\newtheorem{rem}[foo]{Remark}

\swapnumbers

\newcommand{\p}{\partial}
\newcommand{\bp}{\overline{\partial}}

\newcommand{\Ric}{\text{Ric}}
\newcommand{\Rm}{\text{Rm}}

\newcommand{\RR}{\mathbb{R}}

\usepackage{hyperref}
\usepackage{cleveref}

\numberwithin{foo}{section}
\numberwithin{equation}{section}

\title{Independence of Singularity Type for Numerically Effective K\"ahler-Ricci Flows}
\author{Hosea Wondo and Zhou Zhang}

\begin{document}
\date{\today}

\maketitle

\begin{abstract}
    In this paper, we show that the singularity type of solutions to the K\"aher-Ricci flow on a numerically effective manifold does not depend on the initial metric. More precisely, if there exists a type III solution to the K\"ahler-Ricci flow, then any other solution starting from a different initial metric will also be Type III. This generalises a previous result by Y. Zhang for the semi-ample case and confirms a conjecture by V. Tosatti.  
\end{abstract}

\section{Introduction}

\blfootnote{2020 Mathematics Subject Classification. Primary 53C55; Secondary 32Q15.}

The K\"ahler-Ricci flow is a geometric evolution equation used to study K\"ahler geometry. It has proven to be a powerful tool in constructing canonical metrics such as K\"ahler-Einstein metrics \cites{hdC86,PSS06,MS09,PS06,T12}. 

A primary research direction is the study of the analytic minimal model program in birational geometry \cite{M02}, where the flow takes a central role in producing minimal models. Initiated by Song and Tian in \cite{ST06}, the program proposes to emulate the procedure carried out by Hamilton and Perelmann in resolving the Poincare conjecture. The idea is to run the flow on algebraic manifolds and perform `algebraic surgeries' such as flips and contractions each time the flow encounters a singularity \cites{SW13,SW14,SY12}. After finite iterations of this procedure, the resulting manifold would be numerically effective (nef) and hence a minimal model. At this stage, the normalised version of the flow will evolve the metric to a K\"ahler-Einstein metric or a generalised K\"ahler-Einstein metric. Significant progress has been made in this direction; see \cites{ST12,ST16,ST17,T19,GSW16,FzZ15,ysZ19,TWY18,TzlZ16}. 

Let $X$ be a compact K\"ahler manifold with numerically effective line bundle, which corresponds to the final step of the analytic minimal model program. A family on K\"ahler metrics $\omega(t)$ is said to evolve by the normalised K\"ahler-Ricci flow if   
\begin{equation}\label{krf}
\left\lbrace \begin{aligned}
    \pdv{}{t}\omega(t)  &= - \Ric(\omega(t))- \omega(t), \\
    \omega(0) &= \omega_0.
\end{aligned} \right.
\end{equation}

Since the flow is weakly parabolic, we have access to interior estimates, which smooths the metric for short times. In the long run, non-linearities make the metric non-positive, forming a singularity. The precise time of singularity formation is determined by the class of the evolving metric in the cohomology given by 
\begin{equation}
H_{\bar{\partial}}^{1,1}(M, \mathbb{R})=\frac{\{\bar{\partial} \text {-closed real }(1,1) \text {-forms }\}}{\{\bar{\partial} \text {-exact real }(1,1) \text {-forms }\}}.
\end{equation}
The singular time, $T$, is characterised by Tian and the second author in \cite{TzZ06}: 
\begin{equation}\label{maxT}
T:=\sup \left\{t \mid  [\omega(t)] \text{ is K\"ahler}\right\}.
\end{equation}
By formally taking the cohomology class of \eqref{krf} and solving the resulting ODE, we obtain solution
\begin{equation}
    [\omega(t)] = e^{-t}[\omega_0] - (1-e^{-t})c_1(X). 
\end{equation}
Since $X$ is numerically effective, $-c_1(X)$ lies in the closure of the K\"ahler cone, the set of classes with a K\"ahler representative, and thus, the solution exists for all time. 
 Furthermore, for any initial metric, the cohomology class converges to the first Chern class, $-c_1(X)$ in the limit as $t \rightarrow \infty$. This is independent of any starting metric $\omega_0$. \\

The singularity formation reveals the manifold's underlying topology; therefore, the curvature blow-up should be independent of the starting metric. To describe the curvature blow-up, we utilise the following terminology traditionally used in parabolic geometric flows, such as the Ricci flow and Mean Curvature flow. 
\begin{defn}
We say that a long-time solution to the K\"ahler-Ricci flow develops a type III singularity if 
\begin{equation}
\sup _{X \times[0, \infty)} |\Rm(\omega(t))|_{\omega(t)}<+\infty,
\end{equation}
and a type IIb singularity if 
\begin{equation}
\sup _{X \times[0, \infty)} |\Rm(\omega(t))|_{\omega(t)}=+\infty .
\end{equation}
\end{defn}

The precise conjecture for metric independence is stated by  Tossati in \cite{T15Kawa}. \\

\begin{conj}[Tosatti \cite{T15Kawa}]\label{conj}
Let $X$ be a compact K\"ahler manifold with $K_X$ nef, so every solution of the K\"ahler-Ricci flow exists for all positive time. Then the singularity type at infinity does not depend on the choice of the initial metric $\omega_0$.
\end{conj}

In the semi-ample case, Tossati-YG. Zhang in \cite{TWY18} gives an almost complete classification of singularity type based on the topology of $X$. More precisely, they show the following.
\begin{thm}[Tosatti-YG. Zhang \cite{TygZ15}]\label{T0}
Let X be a compact K\"ahler manifold with semi-ample $K_X$ and consider a solution of the K\"ahler-Ricci flow.
\begin{enumerate}
    \item Suppose $\operatorname{kod}(X)=0$.
    \begin{itemize}
        \item If $X$ is a finite quotient of a torus, then the solution is of type III.
        \item If $K_X$ is not a finite quotient of a torus, then the solution is of type IIb.
    \end{itemize}
    \item  Suppose $\operatorname{kod}(X)=n$.
    \begin{itemize}
        \item If $K_X$ is ample, then the solution is type III.
        \item If $K_X$ is not ample, then the solution is type IIb.
    \end{itemize}
    \item  Suppose $0<\operatorname{kod}(X)<n$.
    \begin{itemize}
        \item If $X_y$ is not a finite quotient of a torus, then the solution is of type IIb.
        \item If $X_y$ is a finite quotient of a torus and $V=\emptyset$, then the solution is of type III.
    \end{itemize}
    \item Suppose $n=2$ and $\operatorname{kod}(X)=1$, then the solution is of type III if and only if the only singular fibers on $f$ are of type $m I_0, m>1$.
\end{enumerate}
\end{thm}

An alternative proof to items (1), (2) and a special case of (3) was proven by Y. Zhang in \cite{ysZ20}, where he deduces these results by setting up a comparison of some well-behaved metric and an arbitrary metric. Moreover, he showed that for semiample $K_X$, the singularity type does not depend on the initial metric, thereby partially confirming Conjecture \ref{conj}.   Furthermore, Fong and Y. Zhang in \cite{FysZ20} showed that the set of singular fibers of the semi-ample fibration on which the Riemann curvature blows up at time-infinity is independent of the choice of the initial K\"ahler metric. For more results pertaining to the curvature behaviour of semiample K\"ahler-Ricci flows, see \cites{zZ09,ST16,J20}.

In this paper, we prove Conjecture \ref{conj} without assuming $K_X$ is semiample. More precisely, we show that if a solution develops a type III singularity on a numerically effective manifold, then every other solution develops a type III singularity. Consequently, if a solution develops a type IIb singularity, then any other solution must be type IIb. 

\begin{thm}\label{T1}
Let $\widetilde{\omega}(t)$ be a solution to the normalised K\"ahler-Ricci flow \eqref{krf} on a manifold $X$ with numerically effective $K_X$. 
\begin{equation}
    |\text{{\normalfont Rm}}(\widetilde{\omega}(t))|^2_{\widetilde{\omega}(t)}\leqslant C_0. 
\end{equation}
Then for any other solution to the K\"ahler-Ricci flow, $\omega(t)$, there exists a constant $C>0$ such that 
\begin{equation}\label{T1equiv}
    C^{-1} \widetilde{\omega}(t) \leqslant \omega(t) \leqslant C \widetilde{\omega}(t),
\end{equation}
for all $t \geqslant 0$. 
\begin{equation}\label{ThmRmbdd}
    |\text{{\normalfont Rm}}  (\omega(t))|^2_{\omega(t)}\leqslant C. 
\end{equation}
\end{thm}

\begin{rem}
   Any solution to the K\"ahler-Ricci flow on a numerically effective manifold exists for all time. Indeed, the estimate \eqref{ThmRmbdd} is impossible for finite time singularity due to Z. Zhang's result \cite{zZ10}.
\end{rem}

The proof builds on the results in \cite{ysZ20}, but the major difference is that no assumptions on the semi-ampleness of $X$ need to be made. 



We end the introduction by outlining the structure of this paper. In Section 2, we show a special case of Theorem \ref{T1} when the two initial metrics $\omega_0$ and $\widetilde{\omega}_0$ are related by a simple scaling factor $\lambda_0>0$. We first derive a Monge–Ampère equation where the potential function relates the two evolving metrics under rescaling of space and time. By deriving bounds on the potential function and because of the curvature assumption on $\widetilde{\omega}$, we derive metric equivalence between the known Type III metric $\widetilde{\omega}$ and the arbitrary metric $\omega$. Curvature estimates are then obtained as done in \cite{ysZ20}. In Section 3, we utilise Lemma \ref{2L} to prove Theorem \ref{T1}. Lemma \ref{2L} plays a vital role in constructing ``comparison" solutions to the arbitrary solution $\omega(t)$, which will be needed to obtain a favourable sign when carrying out maximum principle arguments. As a result, a potential function bound is derived and used to obtain metric equivalence. With a similar argument as carried out in Section 2, curvature bounds are obtained for the general theorem. Finally, in Section 4, we utilise Theorem \ref{T1} for examples from \cite{ysZ20} in the numerically effective case. \\

\textbf{Acknowledgement}: 
The authors would like to thank Valentino Tosatti, Feng Wang,  Yashan Zhang, and Zhenlei Zhang for their interest and feedback on this work. They would also like to thank the School of Mathematics and Statistics at The University of Sydney for providing a great study and research environment.

\section{Preliminary Results}
In the first part of this section, we restrict ourselves to the special case of Theorems \ref{T1}, which will be utilised in the proof of the general theorem. More precisely, we show that if there exists a solution to the K\"ahler-Ricci Flow on a numerically effective manifold $X$ with uniformly bounded curvature, then any scaling of the initial metric produces a solution that also evolves with uniformly bounded curvature. \\

In the second part of this section, we recall the proof in \cite{ysZ19}, which allows us to derive a bound on the curvature of a flow $\omega(t)$, given that it is equivalent to a Type III flow $\widetilde{\omega}(t)$.

\subsection{Flow from a Scaled Initial Metric}
We show the following. 
\begin{lem}\label{2L}
Let $X$ be a numerically effective K\"ahler manifolds such that there exists  $\widetilde{\omega}(t)$ a solution to the K\"ahler-Ricci flow with bounded curvature tensor;
\begin{equation}\label{2L2As}
    |\Rm(\widetilde{\omega}(t))|_{\widetilde{\omega}(t)}^2 \leqslant C_0.
\end{equation}
Then for any $\lambda_0>0$, there exists $C>0$ such that the solution $\omega(t)$ starting from $\omega(0) = \lambda_0 \widetilde{\omega}(0)$ satisfies 
\begin{equation}\label{2L2eq}
    C^{-1} \widetilde{\omega}(t) \leqslant \omega(t) \leqslant C \widetilde{\omega} (t),
\end{equation}
\begin{equation}\label{2L2eqCurv}
|\Rm(\omega(t))|_{\omega(t)}^2 \leqslant C,
\end{equation}
for all $t \geqslant 0$.
\end{lem}

We first reduce the problem to a Monge–Ampère equation. The reduction will differ slightly from what is usually done in the literature. Since the initial metrics are equal up to some scaling factor, we scale the solution in space and time by $\lambda(t)$ and $\tau(t)$ respectively so that 
\begin{equation}\label{2rescale}
    [\omega(t)] = \lambda(t)[\widetilde{\omega}(\tau(t))],
\end{equation}
for all $t \geqslant 0$ and $\omega_0 = \lambda_0 \widetilde{\omega}_0$ where $\lambda_0 >0$. For any solution to the K\"ahler-Ricci Flow, the Cohomology class evolves as 
\begin{equation}\label{2evolcoho}
    [\omega(t)] = e^{-t}[\omega_0] - (1-e^{-t}) c_1(X).  
\end{equation}
Then from \eqref{2rescale}, we have 
$$ \lambda_0 e^{-t} \widetilde{\omega}_0 - (1-e^{-t}) c_1(X) = \lambda(t) e^{-\tau(t)} \widetilde{\omega}_0 - (1-e^{-\tau(t)}) \lambda(t) c_1(X), $$
from which we deduce that 
$$
\left\lbrace 
    \begin{aligned}
        \lambda(t) e^{-\tau(t)}  &= \lambda_0 e^{-t}, \\
        \lambda(t) - e^{-\tau(t)}\lambda(t) &= 1 - e^{-t}. 
    \end{aligned}
\right.
$$
Solving this system yields
\begin{equation}
    \lambda(t) = e^{-t}(\lambda_0 - 1) +1,
\end{equation}
and 
\begin{equation}\label{2tau}
    \tau(t) = t + \ln\left( \frac{\lambda(t)}{\lambda_0} \right).
\end{equation}
Furthermore, we check that at $t=0$, the initial class satisfies $[\omega_0] = \lambda_0 [\widetilde{\omega}_0]$. 

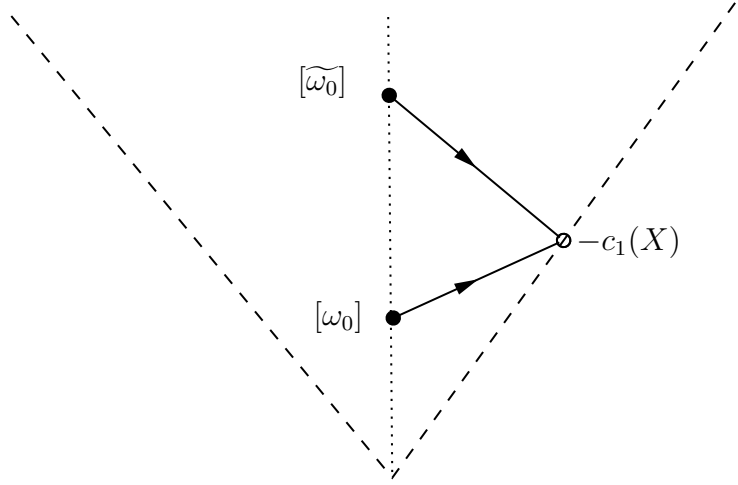
\begin{figure}
    \centering

    \tikzset{every picture/.style={line width=0.75pt}} 
    
    \begin{tikzpicture}[x=0.75pt,y=0.75pt,yscale=-1,xscale=1]
    
    \draw  [dash pattern={on 4.5pt off 4.5pt}] (499.57,44.48) -- (328.51,283.4) -- (137.51,49.96) ;
    \draw    (327,91) -- (412.2,162.49) ;
    \draw [shift={(414,164)}, rotate = 40] [color={rgb, 255:red, 0; green, 0; blue, 0 }  ][line width=0.75]      (0, 0) circle [x radius= 3.35, y radius= 3.35]   ;
    \draw [shift={(370.5,127.5)}, rotate = 220] [fill={rgb, 255:red, 0; green, 0; blue, 0 }  ][line width=0.08]  [draw opacity=0] (12,-3) -- (0,0) -- (12,3) -- cycle    ;
    \draw [shift={(327,91)}, rotate = 40] [color={rgb, 255:red, 0; green, 0; blue, 0 }  ][fill={rgb, 255:red, 0; green, 0; blue, 0 }  ][line width=0.75]      (0, 0) circle [x radius= 3.35, y radius= 3.35]   ;
    \draw    (329,203) -- (411.86,164.98) ;
    \draw [shift={(414,164)}, rotate = 335.35] [color={rgb, 255:red, 0; green, 0; blue, 0 }  ][line width=0.75]      (0, 0) circle [x radius= 3.35, y radius= 3.35]   ;
    \draw [shift={(371.5,183.5)}, rotate = 155.35] [fill={rgb, 255:red, 0; green, 0; blue, 0 }  ][line width=0.08]  [draw opacity=0] (12,-3) -- (0,0) -- (12,3) -- cycle    ;
    \draw [shift={(329,203)}, rotate = 335.35] [color={rgb, 255:red, 0; green, 0; blue, 0 }  ][fill={rgb, 255:red, 0; green, 0; blue, 0 }  ][line width=0.75]      (0, 0) circle [x radius= 3.35, y radius= 3.35]   ;
    \draw  [dash pattern={on 0.84pt off 2.51pt}]  (326.48,51.59) -- (328.51,283.4) ;
    
    \draw (419,155.4) node [anchor=north west][inner sep=0.75pt]    {$-c_{1}( X)$};
    \draw (287.72,193.04) node [anchor=north west][inner sep=0.75pt]    {$[ \omega _{0}]$};
    \draw (279.72,75.04) node [anchor=north west][inner sep=0.75pt]    {$\left[\widetilde{\omega _{0}}\right]$};

    \end{tikzpicture}
    
    \caption{Scaling of Initial Metric}
    \label{fig:1}
\end{figure}

We reduce the flow equation to a Monge–Ampère equation where the rescaled metric takes the role of the reference metric. Using  \eqref{2rescale}, there exists $u:X \rightarrow \RR$ such that 
\begin{equation}\label{2metric}
    \omega(t) = \lambda(t) \widetilde{\omega}(\tau(t)) + \sqrt{-1} \p \bp u(t).
\end{equation}
For ease of notation, we drop the $t$ dependence for $\tau$ and $\lambda$. Taking a $t$-time derivative, 
$$\pdv{\omega}{t}(t) = \lambda'(t) \widetilde{\omega}(\tau) + \lambda \pdv{\tau}{t} \pdv{\widetilde{\omega}}{\tau} + \sqrt{-1} \p \bp  \left( \pdv{u}{t} \right).$$
Using $\lambda' + \lambda = 1$ and $\tau' = \lambda^{-1}$ and \eqref{krf} in $\tau$, the above expression becomes
$$ 
\begin{aligned}
    \pdv{\omega}{t}(t) = \lambda'(t) \widetilde{\omega}(\tau) + \sqrt{-1} \p \bp \log \widetilde{\omega}(\tau)^n - \widetilde{\omega}(\tau) + \sqrt{-1} \p \bp  \left( \pdv{u}{t} \right).
\end{aligned}
$$
On the other hand, \eqref{krf} in $t$ is  
$$ \pdv{\omega}{t}(t) = \Ric(\omega(t)) - \omega(t) = \sqrt{-1}\p \bp \log \omega(t)^n - \lambda \widetilde{\omega}(\tau) - \sqrt{-1} \p \bp u(t). $$
Combining the previous two equations and using $\lambda' + \lambda = 1$ leads to 
$$
\sqrt{-1} \p \bp \left( \pdv{u}{t}\right) = \sqrt{-1} \p \bp \log \omega(t)^n - \sqrt{-1} \p \bp  \log \widetilde{\omega}(\tau)^n = \sqrt{-1} \p \bp u(t). 
$$
Integrating over the compact manifold $X$ yields the complex Monge–Ampère equation 
\begin{equation}\label{2MA}
\left\lbrace 
\begin{aligned}
    \pdv{u}{t} &= \log \left( \frac{\omega(t)^n}{\widetilde{\omega}(\tau)^n} \right) - u(t), \\
    u(0) &= 0.
\end{aligned}\right.
\end{equation}

At $t=0$, we have 
\begin{equation}
    \left.\pdv{u}{t}\right|_{t=0} = n \log(\lambda_0),  
\end{equation}
which is independent of $x \in X$. Since $u(0) \equiv 0$, then the function $u(t)$ is constant over $X$, thus \eqref{2MA} is an ordinary differential equation and \eqref{2metric} is 
\begin{equation}\label{2metequivttau}
    \omega(t) = \lambda(t)\widetilde{\omega}(\tau).
\end{equation}

\begin{proof}[Proof of Lemma \ref{2L}]
    Equation \eqref{2metequivttau} implies that there exist $C>0$ such that  
\begin{equation}\label{2metricequivtaut}
    C^{-1} \widetilde{\omega}(\tau) \leqslant \omega (t) \leqslant C \widetilde{\omega}(\tau). 
\end{equation}
Thus to show \eqref{2L2eq}, we need to show that for some $C>0$
\begin{equation}\label{2tautequiv}
    C^{-1} \widetilde{\omega}(t) \leqslant \widetilde{\omega}(\tau) \leqslant C \widetilde{\omega}(t)
\end{equation}
for all $t \geqslant 0$. To show the above, we use \eqref{2L2As} to find a $C_0>0$ such that 
\begin{equation}
     -C_0 \widetilde{\omega} \leqslant \Ric(\widetilde{\omega}) \leqslant C_0 \widetilde{\omega}.
\end{equation}
This implies that \eqref{krf} can be estimated by 
\begin{equation}\label{2flowest}
    -C_0 \widetilde{\omega} \leqslant \pdv{\widetilde{\omega}}{t} + \widetilde{\omega} \leqslant C_0 \widetilde{\omega}.
\end{equation}
The lower bound in \eqref{2flowest} implies 
$$
    0 \leqslant \pdv{}{t} \left( e^{(C_0+1)t} \widetilde{\omega}\right) 
$$
and the upper bound in \eqref{2flowest} implies 
$$
    \pdv{}{t} \left( e^{(1-C_0)t} \widetilde{\omega}\right) \leqslant 0
$$
for all $t \geqslant 0$. 

We now take cases. Suppose that $0< \lambda_0 < 1$ and assume that $t \geqslant T$ for $T$ large enough such that $\tau  = t + \ln(\lambda(t)/\lambda_0) > t$. Using that $ e^{(1-C_0)t} \widetilde{\omega}(t)$ is decreasing in $t$, we have 
$$
\widetilde{\omega}(\tau) \leqslant e^{(1-C_0)(t-\tau)} \widetilde\omega(t) = e^{C_0-1}\frac{\lambda(t)}{\lambda_0} \widetilde{\omega}(t) \leqslant C \widetilde{\omega}(t),
$$
where $C>0$. On the other hand, $e^{(C_0+1)t} \widetilde{\omega}(t)$ is increasing in $t$ which implies 
$$
\widetilde{\omega}(\tau) \geqslant e^{(C_0+1)(t-\tau)} \widetilde{\omega}(t) = e^{-(C_0+1)}\frac{\lambda(t)}{\lambda_0} \widetilde{\omega}(t) \geqslant C^{-1} \widetilde{\omega}(t)
$$
for some $C>0$. Combining these gives us \eqref{2tautequiv}. If $\lambda_0 > 1$ we can derive the estimate using the same method with $\tau  < t$. 

Combining \eqref{2tautequiv} with \eqref{2metricequivtaut} yields \eqref{2L2eq} for all $t \geqslant T$, for some large $T>0$. Interior estimates imply gives \eqref{2L2eq} for all $t \geqslant 0$. 

The curvature bounds, \eqref{2L2eqCurv}, follow immediately from \eqref{2metequivttau} and the scaling properties of the curvature tensor. 
\end{proof}

\subsection{General Result for Higher Order Estimates}
Using the metric equivalence, we derive higher-order estimates for $\omega(t)$ given that the curvature tensor is bounded for $\widetilde{\omega}$. The method is identical to that by Y. Zhang in \cite{ysZ20}, for convenience we sketch the argument here. When the evolving metric, $\omega(t)$, is equivalent to a fixed metric, $\omega_0$, see Song and Weinkove's notes \cite{SW13Intro}. 

Since we have shown the inequality \eqref{2L2eq}, we consider the two metrics with the same time scale $t$ and drop the dependence of the metrics on $t$ for notational convenience. \\

We first show the following Lemma. 

\begin{lem}
Suppose that $\omega$ and $\widetilde{\omega}$ are solutions to \eqref{krf} and $|\Rm(\widetilde{\omega})|_{\widetilde{\omega}}\leqslant C$. The evolution of the trace satisfies 
        \begin{equation}\label{2traceev}
        \left(\pdv{}{t}-\Delta_{\omega}\right) \tr_{\omega} \widetilde{\omega} \leqslant 
        C \tr_{\omega} \widetilde{\omega}
        +C\left(\tr_{\omega} \widetilde{\omega}\right)^2
        -g^{\bar{j} i}  g^{\bar{q} p}  \widetilde{g}^{\bar{b} a} \nabla_i \widetilde{g}_{p \bar{b}} \nabla_{\bar{j}} \widetilde{g}_{a \bar{q}}.
    \end{equation}
\end{lem}

\begin{proof}
The time derivative is given by 
\begin{equation}
\begin{aligned}
    \pdv{}{t}\left( \tr_{\omega} \widetilde{\omega} \right) &= \pdv{g^{\bar{j}i}}{t}  \widetilde{g}_{i \bar{j}} + g^{\bar{j}i} \pdv{\widetilde{g}_{i \bar{j}}}{t}  \\ 
    &= R^{\bar{j}i}\widetilde{g}_{i \bar{j}} - \tr_{\omega} \Ric(\widetilde{\omega}) 
\end{aligned}
\end{equation}

The laplacian of the trace is given by the standard calculation carried out in \cites{Y78Schawrtz,L67}.

\begin{equation}
     \Delta_{\omega}\tr_{\omega} \widetilde{\omega} = 
     R^{\bar{j}i}\widetilde{g}_{i \bar{j}}  
     - g^{\bar{j} i} g^{\bar{q} p}  \widetilde{R}_{i \bar{j} p \bar{q}} 
     + g^{\bar{j} i} g^{\bar{q} p} \widetilde{g}^{\bar{b} a}  \nabla_i \widetilde{g}_{p \bar{b}} \nabla_{\bar{j}} \widetilde{g}_{a \bar{q}}
\end{equation}
Combining the previous two equations yields 
    \begin{equation}\label{2traceevol}
    \left(\pdv{}{t}-\Delta_{\omega(t)}\right) \tr_{\omega} \widetilde{\omega} =  - \tr_{\omega}\Ric(\widetilde{\omega}) + g^{\bar{j} i} g^{\bar{q} p} \widetilde{R}_{i \bar{j} p \bar{q}}-g^{\bar{j} i} g^{\bar{q} p} \widetilde{g}^{\bar{b} a} \nabla_i \widetilde{g}_{p \bar{b}} \nabla_{\bar{j}} \widetilde{g}_{a \bar{q}}.
    \end{equation}
    The desired inequality follows from assumption \eqref{2L2As}.
\end{proof}

\begin{prop}[[Y. Zhang \cite{ysZ20}]\label{2propcurvbdd}
    Let $\widetilde{\omega}(t)$ be a type III solution to the normalised K\"ahler-Ricci flow and $\omega(t)$ be a arbitrary solution such that for some $C_0>0$,
    \begin{equation}
        C_0^{-1}\widetilde{\omega}(t) \leqslant \omega(t) \leqslant C_0 \widetilde{\omega}(t), 
    \end{equation}
    for all $t \geqslant 0$. Then there exists $C>0$ such that 
    \begin{equation}
        |\Rm(\omega (t))|_{\omega(t)} \leqslant C,
    \end{equation}
    for all $t \geqslant 0$.
\end{prop}

\begin{proof}

Let $\Psi=\left(\Psi_{i j}^k\right)$ where $\Psi_{i j}^k:=\Gamma_{i j}^k-\widetilde{\Gamma}_{i j}^k$ and $S = |\Psi|_\omega^2$. In local coordinates, we have 
\begin{equation}
S=g^{\bar{j} i} g^{\bar{l} k} g^{\bar{q} p} \widetilde{\nabla}_i g_{k \bar{q}} \overline{\widetilde{\nabla}_j g_{l \bar{p}}} .
\end{equation}
For the last term in \eqref{2traceev}, we have   
\begin{equation}
\begin{aligned}
g^{\bar{j} i} g^{\bar{q} p} \widetilde{g}^{\bar{b} a} \nabla_i \widetilde{g}_{p \bar{b}} \nabla_{\bar{j}} \widetilde{g}_{a \bar{q}} & =g^{\bar{j} i} g^{\bar{q} p} \widetilde{g}^{\bar{b} a}\left(\nabla_i-\widetilde{\nabla}_i\right) \widetilde{g}_{p \bar{b}}\left(\nabla_{\bar{j}}-\widetilde{\nabla}_{\bar{j}}\right) \widetilde{g}_{a \bar{q}} \\
& =g^{\bar{j} i} g^{\bar{q} p} \widetilde{g}^{\bar{b} a}\left(-\Psi_{i p}^d\right) \widetilde{g}_{d \bar{b}}\left(-\overline{\Psi_{j q}^e}\right) \widetilde{g}_{a \bar{e}} \\
& \geqslant C^{-1} S .
\end{aligned}
\end{equation}
Hence, we have the estimate 
\begin{equation}\label{2traceest}
\left(\pdv{}{t}-\Delta_\omega\right) t r_\omega \widetilde{\omega} \leqslant C-C^{-1} S.
\end{equation}
The evolution of the tensor $S$ is given by 
\begin{equation}\label{2S}
\left(\partial_t-\Delta_\omega\right) S=S-|\nabla \Psi|_\omega^2-|\overline{\nabla} \Psi|_\omega^2 + \langle \widetilde{\nabla}_i \widetilde{R}_p^k - \Psi * R m(\widetilde{\omega})- g^{\bar{b} a} \widetilde{\nabla}_a \widetilde{R}_{i \bar{b} p}^k, \Psi \rangle_{\omega},
\end{equation}
where $\langle \cdot , \cdot \rangle_\omega$ is the tensor inner product induced by the metric $g(t)$ associated with $\omega(t)$. Following Hamilton's argument in \cite{H82} (or see \cite{SW13Intro}), the assumption \eqref{2L2As} on the curvature tensor of $\widetilde{\omega}$ implies 
$$\left|\nabla_{\widetilde{\omega}(t)} R m(\widetilde{\omega}(t))\right|_{\widetilde{\omega}(t)} \leqslant C,$$
for some $C>0$. In light of this, we estimate \eqref{2S} by
\begin{equation}\label{2Sev}
\left(\pdv{}{t}-\Delta_\omega\right) S \leqslant C S+C-|\nabla \Psi|_\omega^2-|\overline{\nabla} \Psi|_\omega^2.
\end{equation}
We now have all the pieces to set up a maximum principle argument. Let $Q:=S + A \tr_{\omega} \widetilde{\omega}$ for some sufficiently large $A>0$. Using \eqref{2traceest} and \eqref{2Sev}, there exists  $C>0$ such that 
\begin{equation}
\left(\pdv{}{t}-\Delta_\omega\right)Q \leqslant -S+C .
\end{equation}
By a standard maximum principle argument and noting that  $\tr_{\omega(t)} \widetilde{\omega}(t) \leqslant C$, we find a constant $C \geqslant 1$ such that $S \leqslant C.$ 

The last term in \eqref{2Sev} can be estimated by 
\begin{equation}
|\overline{\nabla} \Psi|_\omega^2=\left|\widetilde{R}_{i \overline{b} p}^k-R_{i \overline{b} p}^k\right|_\omega^2 \geqslant \frac{1}{2}|R m(\omega)|_\omega^2-C,
\end{equation}
thus 
\begin{equation}
\left(\pdv{}{t}-\Delta_\omega\right) S \leqslant C-\frac{1}{2}|R m(\omega)|_\omega^2 .
\end{equation}
The evolution of curvature can be estimated by  
\begin{equation}
\left(\pdv{}{t}-\Delta_\omega\right)|R m(\omega)|_\omega \leqslant C|R m(\omega)|_\omega^2-\frac{1}{2}|R m(\omega)|_\omega.
\end{equation}
Let $Q := |\Rm(\omega)|_{\omega} + A S$ for some large  constant $A>0$, we combine the previous two estimates to obtain 
\begin{equation}
\left(\pdv{}{t}-\Delta_\omega\right)\left(|R m(\omega)|_\omega+A S\right) \leqslant-\frac{1}{2}|R m(\omega)|_\omega+C.
\end{equation}

Then by the maximum principle, there exists $C>0$ such that 
\begin{equation}
\sup _{X \times[0, \infty)}|R m(\omega)|_\omega \leqslant C.
\end{equation}
\end{proof}

\section{Proof of Theorem \ref{T1}}

Using Lemma \ref{2L}, we scale a Type III solution $\widetilde{\omega}(t)$ twice to $\omega^+(t)$ and $\omega^-(t)$, such that $\omega_0^- := \lambda_0^- \widetilde{\omega_0} \leqslant \widetilde{\omega}_0 \leqslant \lambda_0^+ \widetilde{\omega}_0 :=  \omega_0^+$. We then choose $\lambda_0^-$ small enough and $\lambda_0^+$ large enough, such that $\omega_0^{-} \leqslant \omega_0 \leqslant \omega_0^+$. Furthermore, these solutions starting from these scaled metrics are type III. We treat these new solutions act as reference metrics, where either $\omega^+(t)$ or $\omega^-(t)$ is chosen to obtain a favourable sign when carrying out maximum principle arguments.  

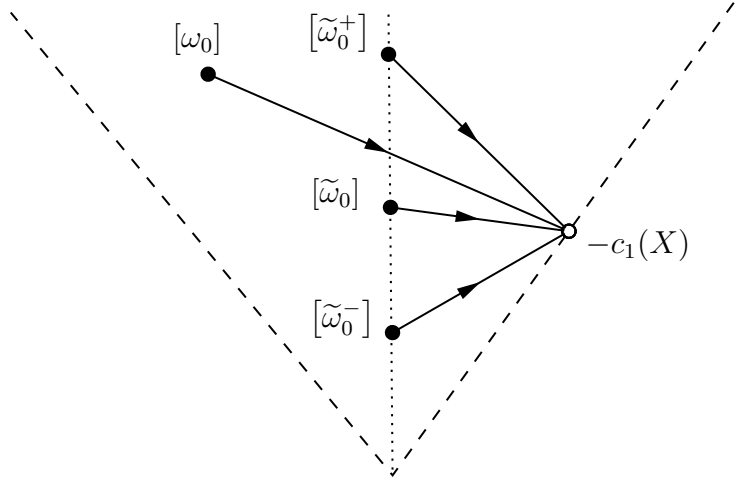
\begin{figure}
    \centering

    \tikzset{every picture/.style={line width=0.75pt}} 
    
    \begin{tikzpicture}[x=0.75pt,y=0.75pt,yscale=-1,xscale=1]
    
    \draw  [dash pattern={on 4.5pt off 4.5pt}] (498.25,45.06) -- (328.03,282.81) -- (137.51,49.96) ;
    \draw    (326,71) -- (414.33,158.35) ;
    \draw [shift={(416,160)}, rotate = 44.68] [color={rgb, 255:red, 0; green, 0; blue, 0 }  ][line width=0.75]      (0, 0) circle [x radius= 3.35, y radius= 3.35]   ;
    \draw [shift={(371,115.5)}, rotate = 224.68] [fill={rgb, 255:red, 0; green, 0; blue, 0 }  ][line width=0.08]  [draw opacity=0] (12,-3) -- (0,0) -- (12,3) -- cycle    ;
    \draw [shift={(326,71)}, rotate = 44.68] [color={rgb, 255:red, 0; green, 0; blue, 0 }  ][fill={rgb, 255:red, 0; green, 0; blue, 0 }  ][line width=0.75]      (0, 0) circle [x radius= 3.35, y radius= 3.35]   ;
    \draw    (328,211) -- (413.97,161.18) ;
    \draw [shift={(416,160)}, rotate = 329.91] [color={rgb, 255:red, 0; green, 0; blue, 0 }  ][line width=0.75]      (0, 0) circle [x radius= 3.35, y radius= 3.35]   ;
    \draw [shift={(372,185.5)}, rotate = 149.91] [fill={rgb, 255:red, 0; green, 0; blue, 0 }  ][line width=0.08]  [draw opacity=0] (12,-3) -- (0,0) -- (12,3) -- cycle    ;
    \draw [shift={(328,211)}, rotate = 329.91] [color={rgb, 255:red, 0; green, 0; blue, 0 }  ][fill={rgb, 255:red, 0; green, 0; blue, 0 }  ][line width=0.75]      (0, 0) circle [x radius= 3.35, y radius= 3.35]   ;
    \draw  [dash pattern={on 0.84pt off 2.51pt}]  (326,51) -- (328.03,282.81) ;
    \draw    (236,81) -- (413.85,159.06) ;
    \draw [shift={(416,160)}, rotate = 23.7] [color={rgb, 255:red, 0; green, 0; blue, 0 }  ][line width=0.75]      (0, 0) circle [x radius= 3.35, y radius= 3.35]   ;
    \draw [shift={(326,120.5)}, rotate = 203.7] [fill={rgb, 255:red, 0; green, 0; blue, 0 }  ][line width=0.08]  [draw opacity=0] (12,-3) -- (0,0) -- (12,3) -- cycle    ;
    \draw [shift={(236,81)}, rotate = 23.7] [color={rgb, 255:red, 0; green, 0; blue, 0 }  ][fill={rgb, 255:red, 0; green, 0; blue, 0 }  ][line width=0.75]      (0, 0) circle [x radius= 3.35, y radius= 3.35]   ;
    \draw    (327,148) -- (413.67,159.69) ;
    \draw [shift={(416,160)}, rotate = 7.68] [color={rgb, 255:red, 0; green, 0; blue, 0 }  ][line width=0.75]      (0, 0) circle [x radius= 3.35, y radius= 3.35]   ;
    \draw [shift={(371.5,154)}, rotate = 187.68] [fill={rgb, 255:red, 0; green, 0; blue, 0 }  ][line width=0.08]  [draw opacity=0] (12,-3) -- (0,0) -- (12,3) -- cycle    ;
    \draw [shift={(327,148)}, rotate = 7.68] [color={rgb, 255:red, 0; green, 0; blue, 0 }  ][fill={rgb, 255:red, 0; green, 0; blue, 0 }  ][line width=0.75]      (0, 0) circle [x radius= 3.35, y radius= 3.35]   ;
    
    \draw (423,157.4) node [anchor=north west][inner sep=0.75pt]    {$-c_{1}( X)$};
    \draw (282.72,49.04) node [anchor=north west][inner sep=0.75pt]    {$\left[\widetilde{\omega }_{0}^{+}\right]$};
    \draw (216.72,54.04) node [anchor=north west][inner sep=0.75pt]    {$[ \omega _{0}]$};
    \draw (284.72,193.04) node [anchor=north west][inner sep=0.75pt]    {$\left[\widetilde{\omega }_{0}^{-}\right]$};
    \draw (285.72,132.04) node [anchor=north west][inner sep=0.75pt]    {$\left[\widetilde{\omega }_{0}\right]$};
      
    \end{tikzpicture}

    \caption{Scaling of Type III solutions}
    \label{fig:2}
\end{figure}

\subsection{Reduction to a Monge–Ampère Equation}\label{S3.1}

Let $\chi \in [\omega_\infty]$ be a representative of the limiting class under the K\"ahler-Ricci flow. Using the usual reference metrics, we have 
\begin{equation}\label{3ref}
    \begin{aligned}
        \omega(t) &= e^{-t} \omega_0 + (1-e^{-t}) \chi + \sqrt{-1} \p \bp \varphi, \\
        \omega^+(t) &= e^{-t} \omega^+_0 + (1-e^{-t}) \chi + \sqrt{-1} \p \bp \varphi^+, \quad \text{ and}\\
        \omega^-(t) &= e^{-t} \omega^-_0 + (1-e^{-t}) \chi + \sqrt{-1} \p \bp \varphi^-. \\
    \end{aligned}
\end{equation}

In our proof of Theorem \ref{T1}, we require three Monge-Ampère equations relating pairs of $\omega(t)$, $\omega^+(t)$ and $\omega^-(t)$. Taking a time derivative of $\omega(t)$ and $\omega^-(t)$ in \eqref{3ref}, and using the flow equation, we arrive at a Monge-Ampère equation for $u:= \varphi - \varphi^-$ given by
\begin{equation}\label{3MAu}
    \pdv{u}{t} = \log\left( \frac{\omega^n}{(\omega^-)^n} \right) - u
\end{equation}
where
\begin{equation}
    \omega(t) = \omega^-(t) + e^{-t}(\omega_0 - \omega^-_0) + \sqrt{-1} \p \bp u.
\end{equation}
Equation \eqref{3MAu} will be the primary Monge-Ampère utilised in the proof of the main Theorem. Note that the construction of $\omega^-_0$ ensures $\omega_* := \omega_0 - \omega^-_0 >0.
$ 
Similarly for $\psi := \varphi^+ - \varphi^-$ we have 
\begin{equation}\label{3MApsi}
    \pdv{\psi}{t} = \log \left( \frac{(\omega^+)^n}{(\omega^-)^n} \right) - \psi.
\end{equation}
The function $\psi$ relates the two evolving metrics by 
\begin{equation}
    \omega^+(t) = \omega^-(t) + e^{-t} (\omega_0^+ - \omega^-_0) + \sqrt{-1} \p \bp \psi. 
\end{equation}

Finally, for $v:= \varphi - \varphi^+$, we have the following Monge-Ampere equation 
\begin{equation}\label{3MAv}
    \pdv{v}{t} = \log\left( \frac{\omega^n}{(\omega^+)^n} \right) - v
\end{equation}
where 
\begin{equation}
    \omega(t) = \omega^+(t) + e^{-t}(\omega_0 - \omega^+_0) + \sqrt{-1} \p \bp v.
\end{equation}

\subsection{Potential Estimates}

Our aim is to derive potential estimates for $u$ and its time derivative. We begin by calculating the evolution of $u$ and $\pdv{u}{t}$. As before, we drop the $t$ dependence on the metrics and potential functions for ease of notation. To denote geometric quantities associated with $\omega^+(t)$ and $\omega^-(t)$ with their respective symbol in superscripts. 

\begin{lem}\label{3lemevol}
    For $\omega_* := \omega_0 - \omega^-_0 >0$, any solution $u$ to \eqref{3MAu} satisfies 
    \begin{equation}\label{3lemev1}
        \left(\pdv{}{t} - \Delta_{\omega(t)}\right)u  = \pdv{u}{t} - n + \tr_{\omega} \omega^- + e^{-t} \tr_\omega \omega_*,
    \end{equation}
    and 
        \begin{equation}\label{3lemev2}
        \left(\pdv{}{t} - \Delta_{\omega(t)}\right)\pdv{u}{t}  = R^- - \tr_{\omega} \Ric(\omega^-) - \pdv{u}{t} - \tr_{\omega} \omega^-  - e^{-t} \tr_\omega \omega_* + n.
    \end{equation}
\end{lem}

\begin{proof}
    The first equation, \eqref{3lemev1} follows from \eqref{3ref};
    $$ \left(\pdv{}{t} - \Delta_{\omega(t)}\right)u  = \pdv{u}{t} - \tr_{\omega} (\omega - \omega^- - e^{-t}\omega_*) =  \pdv{u}{t} - n + \tr_{\omega} \omega^- + e^{-t} \tr_\omega \omega_*.
    $$
    Equation \eqref{3lemev2} is derived using \eqref{3MAu}; 
    $$
        \begin{aligned}
            \left(\pdv{}{t} - \Delta_{\omega(t)}\right) \pdv{u}{t} &=  \left(\pdv{}{t} - \Delta_{\omega(t)}\right) \log \left( \frac{\omega^n}{{\omega^-}^n}\right) -  \left(\pdv{}{t} - \Delta_{\omega(t)}\right)u \\
             &= R^- - \tr_{\omega} \Ric(\omega^-) -   \pdv{u}{t} + n - \tr_{\omega} \omega^- - e^{-t}
         \tr_{\omega}\omega_* . 
         \end{aligned}
    $$
\end{proof}
\begin{lem}\label{3pot1}
    Let $u$ be a solution to \eqref{3MAu}. Then there exists a uniform $C>0$ such that 
    \begin{equation}
        -C e^{-t} \leqslant u(t) \leqslant C 
    \end{equation}
    for all $t \geqslant 0$.
\end{lem}
\begin{proof}

    The lower bound follows from the maximum principle. Indeed, at a minimum point of $u$, 
    $$\pdv{}{t} \left( e^t u \right) = e^t \log \left(  \frac{(\omega^-(t) + e^{-t}(\omega_0-\omega_0^-)+\sqrt{-1}\p \bp u)^n}{(\omega^-)^n}\right) \geqslant e^t\log(1) = 0.$$
    Integrating yields 
    $$ u(t) \geqslant - Ce^{-t}.$$
    To derive an upper bound for $u$, we first observe that at a maximum for $v$ in  \eqref{3MAv}, we have 
    $$
    \pdv{v_{\max}}{t} \leqslant \log \left( \frac{(\omega_+ + e^{-t}(\omega_0 - \omega_0^+))^n}{(\omega^+)^n} \right) - v_{\max} \leqslant  - v_{\max}.
    $$
    Then by the maximum principle, we deduce that $v \leqslant 0$. On the other hand, we have $C>0$ such that 
    \begin{equation}\label{psibdd}
        |\psi| \leqslant C 
    \end{equation}
    for all $t>0$. Indeed, the initial metrics $\omega^+_0, \omega^-_0$ and $\widetilde{\omega}_0$ are scalar multiples of each other, Lemma \ref{2L} and \eqref{3MApsi} imply 
$$
\pdv{}{t} \left( e^t \psi \right) = e^t \log \left( \frac{(\omega^+)^n}{(\omega^-)^n} \right) \leqslant Ce^{t}
    $$
for some $C>0$. Integrating the above gives us \eqref{psibdd}. It now follows that  
    $$u(t) = \varphi(t) - \varphi^-(t) = v(t) + \psi(t) \leqslant C  $$
where $C>0$ is independent of time.  
\end{proof}

\begin{lem}\label{3pot2}
    Let $u$ be a solution to \eqref{3MAu}. Then there exists $C>0$ such that 
    \begin{equation}
        \left|\pdv{u}{t}\right| \leqslant C,
    \end{equation}
    for all $t \geqslant 0$.
\end{lem}

    \begin{proof}
     We first note that Lemma \ref{2L} implies that for some $C>0$, we have 
\begin{equation}\label{3meq1}
    C^{-1}\widetilde{\omega} (t) \leqslant \omega^- (t)   \leqslant C \widetilde{\omega} (t),
\end{equation}
and 
\begin{equation}\label{3curbound}
    |\Rm(\omega^-(t))|_{\omega^-(t)}^2 \leqslant C,
\end{equation}
for all $t \geqslant 0$. 
Let 
\begin{equation}
    Q:= \pdv{u}{u} - Au,
\end{equation}
where $A>0$ a constant that will be set later. By combining \eqref{3lemev1} and \eqref{3lemev2}, and using \eqref{3curbound}, we obtain 
    $$
    \begin{aligned}
        \left( \pdv{}{t} - \Delta_{\omega} \right) Q &= R^- - \tr_{\omega} \Ric (\omega^-)  -(A+1) \tr_{\omega}  \omega^- - (A+1)e^{-t}\tr_{\omega} \omega_* - (A+1)\pdv{u}{t}+ (A+1)n \\
        & \leqslant -(A+1-C_0) \tr_{\omega}  \omega^- - (A+1)e^{-t}\tr_{\omega} \omega_* - (A+1)\pdv{u}{t}+ (A+1+C_0)n. 
    \end{aligned}
    $$
    We choose $A = 2+C_0$ to obtain the estimate
    $$
    \left( \pdv{}{t} - \Delta_{\omega} \right) Q  \leqslant - C \pdv{u}{t} +C. 
    $$
    Applying the parabolic maximum principle yields the upper bound 
    $$
        \pdv{u}{t} \leqslant C,
    $$
    for some $C>0$. 
    
    For the lower bound, we instead consider the quantity 
    $$Q:= \pdv{u}{t} +Au,$$
    for some $A >0$ to be determined later. Once again, we use \eqref{3lemev1}, \eqref{3lemev2} and \eqref{3curbound} to set up a maximum principle argument: 
        $$
    \begin{aligned}
        \left( \pdv{}{t} - \Delta_{\omega} \right) Q &= R^-(t) - \tr_{\omega} \Ric (\omega^-) + (A-1)  \tr_{\omega}  \widetilde{\omega} +(A-1)e^{-t}\tr_{\omega}\omega_*  + (A-1)\pdv{u}{t}+ n \left( 1 - A \right)\\
        & \geqslant (C_0+1-A)n +(A-1)\pdv{u}{t}+(A -C_0-1)\tr_{\omega}\omega^- +(A-1)e^{-t}\tr_{\omega}\omega_*.
    \end{aligned}
    $$
    Applying a similar argument to before, we choose $A = 2+C_0$ to obtain an estimate 
    $$
    \begin{aligned}
        \left( \pdv{}{t} - \Delta_{\omega} \right) Q & \geqslant -C + C\pdv{u}{t}  + \tr_{\omega}\omega^- \\
        & \geqslant -C + C \pdv{u}{t} + n \left( \frac{(\omega^-)^n}{\omega^n} \right)^\frac{1}{n} \\
        &= -C + C\pdv{u}{t} +n e^{- \frac{1}{n}\left( \pdv{u}{t} + u \right)}
    \end{aligned}
     $$
     Then applying the maximum principle, using that $u$ is uniformly bounded, shows that there exists a $C>0$ independent of time such that 
     \begin{equation}
         \pdv{u}{t} \geqslant -C.
     \end{equation}
\end{proof}

\subsection{Proof of Theorem \ref{T1}}

We now derive a metric equivalence between the Type III solution $\widetilde{\omega}$ and the arbitrary solution $\omega$.

\begin{proof}
From a standard calculation, using \eqref{3curbound}, we have for some $C>0$ and $C_0>0$
\begin{equation}
   \left(\pdv{}{t} - \Delta_{\omega} \right) \log \operatorname{tr}_\omega \omega^- \leqslant C_0 \operatorname{tr}_\omega \omega^- +C,
\end{equation}
for all times $t \geqslant 0$. 
The uniform bound on $\pdv{u}{t}$ from Lemma \ref{3pot2} implies that 
\begin{equation}
\left(\pdv{}{t} - \Delta_{\omega} \right) u = \pdv{u}{t}  -n + \tr_{\omega} \omega^- + e^{-t} \tr_{\omega} \left( \omega_0-\omega_0^- \right) \geqslant -C + \tr_{\omega} \omega^-.
\end{equation}
Following a similar argument as before, we define $Q: = \log \tr_{\omega} \omega^- - Au $, then using the two inequalities, we choose $A$ large enough such that 
\begin{equation}
    \left(\pdv{}{t} - \Delta_{\omega} \right)Q \leqslant C - \tr_{\omega} \omega^-.
\end{equation}
Applying the maximum principle, and again using Lemmas \ref{3pot1} and \ref{3pot2}, we obtain some $C>0$ such that 
\begin{equation}
    \tr_{\omega} \omega^- \leqslant C,
\end{equation}
for all $t \geqslant 0$. Combining the potential estimates from Lemma \ref{3pot1} and \ref{3pot2}, we derive from the Monge–Ampère equation \eqref{3MAu} the volume bounds \begin{equation}
    C^{-1} \leqslant \frac{\omega^n}{(\omega^-)^n}  \leqslant  C , 
\end{equation}
for some $C>0$ for all $t \geqslant 0$. Similar to before, we employ standard eigenvalue estimates to obtain 
$$\tr_{\omega^{-}} \omega \leqslant n \left( \frac{\omega^n}{(\omega^{-})^n}\right) \tr_{\omega} \omega^{-} \leqslant C.  $$
Then the following metric equivalence follows from the two trace bounds; 
\begin{equation}
    C^{-1} \omega^{-}(t) \leqslant \omega(t) \leqslant C \omega^{-}(t). 
\end{equation}
Then the above inequality and \eqref{3meq1} yields \eqref{T1equiv}. To obtain bounds on the curvature tensor, we apply Proposition \ref{2propcurvbdd} with $\omega^{-}(t)$ in place of $\widetilde{\omega}$. This completes the proof of Theorem \ref{T1}.
\end{proof}

We can now apply Theorem \ref{T1} to answer Conjecture \ref{conj}. 

\begin{thm}
    Let $X$ be a K\"ahler manifold with numerically effective canonical line bundle. Then the singularity type of solutions to the K\"ahler-Ricci flow is independent of the initial metric. 
\end{thm}

\begin{proof}
    Suppose a solution $\widetilde{\omega}$ is Type III. Then by Theorem \ref{T1}, any other solution starting from a different initial metric must also be Type III. On the other hand, if a type IIb solution exists, then any other solution must be Type IIb as otherwise, Theorem \ref{T1} implies that all solutions must be Type III, a contradiction. 
\end{proof}

\begin{rem}
    Finally, we note that Theorem \ref{T1} can be interpreted as evidence for the K\"ahler extension of the abundance conjecture. Indeed, if a counter example to Theorem \ref{T1} could be constructed, then Y. Zhang's result \cite{ysZ20} implies that the manifold cannot be semiample.
\end{rem}

\section{Further Remarks}

Our results and techniques have applications and generalisations in many settings.\\

We start with some examples where an explicit solution to the normalised K\"ahler-Ricci Flow is known. This allows us to deduce the singularity type for an arbitrary initial metric for these manifolds. These examples are obtained from \cite{ysZ19}. 

\begin{enumerate}
    \item \textbf{Calabi-Yau Metrics}: Suppose that $X$ admits a Calabi-Yau metric. Then under the K\"ahler-Ricci flow, we check that $\omega(t) = e^{-t} \omega_{CY}$ is a solution with $\omega(0) = \omega_{CY}$. For such solutions, the curvature evolves as 
\begin{equation}
|\Rm(\omega(t))|_{\omega(t)}=e^t\left|\Rm\left(\omega_{C Y}\right)\right|_{\omega_{C Y}},
\end{equation}
and therefore develops a type III singularity if and only if $\omega_{CY}$ is a flat metric, which is only possible if and only if $X$ is a finite quotient of a torus. Therefore, we conclude from Theorem \ref{T1} that the K\"ahler-Ricci Flow develops a type III solution on a numerically effective manifold $X$ if and only if it is a finite quotient of a torus. Furthermore, this demonstrates Item (1) of Theorem \ref{T0} without assuming $K_X$ is semiample. 

\begin{rem}
    Item (2) in Theorem \ref{T0} assumes $K_X$ is big. Since nef and big imply semiample, the generalisation to nef is immediate. 
\end{rem}
\item \textbf{Product Manifold}: 
Consider the product manifold $X:= B \times Y$ where $Y$ is a Calabi-Yau manifold. This is a special case of item (3). Suppose that $\omega$ is a K\"ahler-Einstein metric on $B$, then we check that 
\begin{equation}
    \widetilde{\omega}(t) = e^{-t} \omega_{CY} + (1-e^{-t})\omega_B
\end{equation}
is a solution to the K\"ahler-Ricci flow. The curvature is given by 
\begin{equation}
    |\Rm(\omega(t)|_{\omega(t)}^2 = e^{2t}|\Rm(\omega_{CY})|_{\omega_{CY}}^2 + \frac{1}{(1-e^{-t})^2}|\Rm(\omega_B)|_{\omega_B}^2.
\end{equation}
Thus, we have a type III singularity if and only if $\omega_{CY}$ is flat which occurs if $Y$ is a finite quotient of a torus. Thus, a type III singularity develops if and only if $Y$ is a finite quotient of a torus. 
Thus, using Theorem \ref{T1}, any initial metric $\omega_0$, there exists $C>0$ such that 
\begin{equation}
   C^{-1} \widetilde{\omega}(t) \leqslant \omega (t)\leqslant C \widetilde{\omega}(t)
\end{equation}
and 
\begin{equation}
    |\Rm(\omega(t))|_{\omega(t)} \leqslant C,
\end{equation}
for all $t \geqslant 0$.
\end{enumerate}

Furthermore, our method can be applied to other flows, for example, the modified K\"ahler-Ricci flow as studied in \cites{zZ09mod, yY11, WzZ22}. More precisely, the independence of infinite time singularity type should follow from metric equivalence and higher order estimates.

\bibliography{ref}

\end{document}